\documentclass[12pt]{amsart}
\usepackage{graphicx}
\usepackage[usenames]{color}
\usepackage{amssymb, amsmath, amsthm, fullpage, xcolor, mathabx}
\usepackage[mathscr]{eucal}
\usepackage{hyperref}
\usepackage{comment}
\newtheorem{theorem}{Theorem}[section]

\newtheorem{lemma}[theorem]{Lemma}

\newtheorem{proposition}[theorem]{Proposition}
\theoremstyle{definition}
\newtheorem{defn}[theorem]{Definition}

\newtheorem{remark}[theorem]{Remark}

\newtheorem*{theoremno}{Theorem}
\newtheorem*{props}{Proposition}

\numberwithin{theorem}{section} 
\numberwithin{equation}{section}

\newcommand{\Z}{\mathbb Z}
\newcommand{\N}{\mathbb N}
\newcommand{\R}{\mathbb R}
\newcommand{\C}{\mathbb C}
\newcommand{\Q}{\mathbb Q}
\newcommand{\bbH}{\mathbb{H}}

\newcommand{\sm}[4]{\left(\begin{smallmatrix}#1&#2\\ #3&#4 \end{smallmatrix} \right)}

\newcommand{\qs}{Q_{\boldsymbol{\zeta_n}}}
\newcommand{\qSubgroup}{\Gamma_{\boldsymbol{\zeta_n}}}
\newcommand{\lcm}{\textnormal{lcm}}
\newcommand{\bs}{\boldsymbol}

\allowdisplaybreaks

\begin{document}

\thanks{Acknowledgements:  The authors thank the Banff International Research Station (BIRS) and the Women in Numbers 4 (WIN4) workshop for the opportunity to initiate this collaboration.  The first author is grateful for the support of National Science Foundation grant DMS-1449679, and the Simons Foundation.}

\title{Quantum modular forms and singular combinatorial series with distinct roots of unity}
\author{Amanda Folsom, Min-Joo Jang, Sam Kimport, and Holly Swisher}

\maketitle
  
 \begin{abstract} Understanding the relationship between mock modular forms and quantum modular forms is a problem of current interest.  Both mock and quantum modular forms exhibit modular-like transformation properties under suitable subgroups of $\textnormal{SL}_2(\mathbb Z)$, up to nontrivial error terms; however, their domains (the upper half-plane $\mathbb H$, and the rationals $\mathbb Q$, respectively) are notably different.  Quantum modular forms, originally defined by Zagier in 2010, have  also been shown to be related to the diverse areas of colored Jones polynomials, meromorphic Jacobi forms, partial theta functions, vertex algebras, and more.  

 In this paper  we study the $(n+1)$-variable combinatorial rank generating function $R_n(x_1,x_2,\dots,x_n;q)$ for $n$-marked Durfee symbols.  These are $n+1$ dimensional multisums for $n>1$, and specialize    to the ordinary two-variable partition rank generating function when $n=1$.  The mock modular properties of $R_n$ when viewed as a function of $\tau\in\mathbb H$, with $q=e^{2\pi i \tau}$,  for various $n$ and fixed parameters $x_1, x_2, \cdots, x_n$, have been studied in a series of papers.  Namely, by Bringmann and Ono when $n=1$ and $x_1$ a root of unity; by Bringmann when $n=2$ and $x_1=x_2=1$; by Bringmann, Garvan, and Mahlburg for $n\geq 2$ and $x_1=x_2=\dots=x_n=1$; and by the first and third authors for $n\geq 2$ and the $x_j$ suitable roots of unity ($1\leq j \leq n$).    

The quantum modular properties of $R_1$ readily follow from existing results.  Here, we focus our attention on the case $n\geq 2$, and prove for any $n\geq 2$ that the combinatorial generating function $R_n$ is a quantum modular form when viewed as a function of $x \in \mathbb Q$, where $q=e^{2\pi i x}$,  and the $x_j$ are suitable distinct roots of unity.    
\end{abstract}
 
\section{Introduction and Statement of results}\label{intro}

  \subsection{Background} 
   
A \emph{partition} of a positive integer $n$ is any non-increasing sum of positive integers that adds to $n$.  Integer partitions and modular forms are beautifully and intricately linked, due to the fact that the generating function for the partition function $p(n):= \# \{\mbox{partitions of } n \}$, is related to Dedekind's eta function $\eta(\tau)$, a weight $\frac12$ modular form defined by
\begin{align}\label{def_eta} 
\eta(\tau) := q^{\frac{1}{24}}\prod_{n=1}^\infty (1-q^n). 
\end{align}
Namely,
\begin{equation}\label{p-eta} 
1 + \sum_{n=1}^\infty p(n)q^n = \frac{1}{(q;q)_{\infty}} = q^{\frac{1}{24}}\eta(\tau)^{-1},
\end{equation}
where here and throughout this section $q:=e^{2\pi i \tau}$, $\tau \in \bbH:= \{x + i y \ | \ x \in \mathbb R, y \in \mathbb R^+\}$ the upper half of the complex plane, and the $q$-Pochhammer symbol is defined for $n\in\N_0\cup\{\infty\}$ by  $$(a)_n=(a;q)_n:=\prod_{j=1}^n (1-aq^{j-1}).$$  
 In fact,   the connections between partitions and modular forms  go much deeper, and one example of this is  given by  the combinatorial rank function.  Dyson \cite{Dyson} defined the {\em rank} of a partition to be its largest part minus its number of parts, and  the \emph{partition rank function} is defined by 
\[
N(m,n) := \# \{\mbox{partitions of } n \mbox{ with rank equal to } m \}.
\]   
For example, $N(7,-2)=2$, because precisely 2 of the 15 partitions of $n=7$ have rank equal to $-2$; these are $2+2+2+1$, and $3+1+1+1+1$.  
     
Partition rank functions have a rich history in the areas of combinatorics, $q$-hypergeometric series, number theory and modular forms.  As one particularly notable example, Dyson conjectured that the rank could be used to combinatorially explain Ramanujan's famous partition congruences modulo 5 and 7;  this conjecture was later proved by Atkin and Swinnerton-Dyer \cite{AtkinSD}.  

It is well-known that the associated two variable generating function for $N(m,n)$ may be expressed as a $q$-hypergeometric series \begin{align}\label{rankgenfn}  \sum_{m=-\infty}^\infty \sum_{n=0}^\infty N(m,n) w^m q^n =   \sum_{n=0}^\infty \frac{q^{n^2}}{(wq;q)_n(w^{-1}q;q)_n} =: R_1(w;q),\end{align}   noting here that $N(m,0)=\delta_{m0}$, where $\delta_{ij}$ is the Kronecker delta function.

Specializations in the $w$-variable of the rank generating function have been of particular interest in the area of modular forms.   For example, when $w= 1$, we have that
\begin{equation}\label{r1mock1}
R_1(1;q) = 1+ \sum_{n=1}^\infty   p(n) q^n = q^{\frac{1}{24}}\eta^{-1}(\tau)
\end{equation}
thus recovering \eqref{p-eta}, which shows that the generating function for $p(n)$ is (essentially\footnote{Here and throughout, as is standard in this subject for simplicity's sake, we may slightly abuse terminology and refer to a function as a modular form or other modular object when in reality it must first be multiplied by a suitable power of $q$ to transform appropriately. 
 }) the reciprocal of a weight $\frac12$ modular form.  

If instead we let $w=-1$, then 
\begin{equation}\label{r1mock2}
R_1(-1;q) =   \sum_{n=0}^\infty \frac{q^{n^2}}{(-q;q)_n^2} =: f(q).
\end{equation}
The function $f(q)$ is not a modular form, but one of Ramanujan's original third order mock theta functions.  

Mock theta functions, and more generally mock modular forms and harmonic Maass forms have been major areas of study.   In particular,  understanding how Ramanujan's mock theta functions fit into the theory of modular forms was a question that persisted from Ramanujan's death in 1920 until  
 the groundbreaking 2002 thesis of Zwegers \cite{Zwegers1}: we now know that Ramanujan's mock theta functions, a finite list of curious $q$-hypergeometric functions including $f(q)$, exhibit suitable modular transformation properties after they are \emph{completed} by the addition of certain nonholomorphic functions.  In particular, Ramanujan's mock theta functions are examples of \emph{mock modular forms}, the holomorphic parts of \emph{harmonic Maass forms}.  Briefly speaking, harmonic Maass forms, originally defined by Bruiner and Funke \cite{BF}, are nonholomorphic generalizations of ordinary modular forms that in addition to satisfying appropriate modular transformations, must be eigenfunctions of a certain weight $k$-Laplacian operator, and satisfy suitable growth conditions at cusps (see \cite{BFOR, BF, OnoCDM, ZagierB} for more).   
 
 Given that specializing $R_1$ at $w=\pm 1$ yields two different modular objects, namely an ordinary modular form and a mock modular form as seen in \eqref{r1mock1} and \eqref{r1mock2}, it is natural to ask about the modular properties of $R_1$ at other values of $w$.  Bringmann and Ono answered this question in \cite{BO}, and used the theory of harmonic Maass forms to prove  that upon specialization of the parameter $w$ to  complex roots of unity not equal to $1$, the rank generating function $R_1$  is also a mock modular form.  (See also \cite{ZagierB} for related work.) 
 
 \begin{theoremno}[\cite{BO} Theorem 1.1]  If $0<a<c$, then $$q^{-\frac{\ell_c}{24}}R_1(\zeta_c^a;q^{\ell_c}) + \frac{i \sin\left(\frac{\pi a}{c}\right) \ell_c^{\frac{1}{2}}}{\sqrt{3}} \int_{-\overline{\tau}}^{i\infty} \frac{\Theta\left(\frac{a}{c};\ell_c \rho\right)}{\sqrt{-i(\tau + \rho)}} d\rho $$ is a harmonic Maass form of weight $\frac{1}{2}$ on $\Gamma_c$.  
 \end{theoremno}
\noindent Here, $\zeta_c^a := e^{\frac{2\pi ia}{c}}$ is a $c$-th root of unity, $\Theta\left(\frac{a}{c};\ell_c\tau\right)$ is a certain weight $3/2$ cusp form, $\ell_c:=\lcm(2c^2,24)$, and $\Gamma_c$ is a particular subgroup of $\textnormal{SL}_2(\mathbb Z)$. 

In this paper,  as well as in prior work of two of the authors \cite{F-K}, we study the related problem of understanding the modular properties of certain combinatorial $q$-hypergeometric series arising from objects called $n$-marked Durfee symbols, originally defined by Andrews in his notable work \cite{Andrews}.  

To understand $n$-marked Durfee symbols, we first describe Durfee symbols. For each partition, the Durfee symbol catalogs the size of its Durfee square, as well as the length of the columns to the right as well as the length of the rows beneath the Durfee square.  For example, 
below we have the partitions of $4$, followed by their Ferrers diagrams with any element belonging to their Durfee squares marked by a square $(\sqbullet)$,  followed by their Durfee symbols.
\[
\begin{array}{ccccc}
4 & 3+1 & 2+2 & 2+1+1 & 1+1+1+1 \\
\begin{array}{lllll}  \sqbullet & \bullet & \bullet & \bullet \end{array} & \begin{array}{lll} \sqbullet & \bullet & \bullet \\ \bullet & &  \end{array}  &    \begin{array}{ll}  \sqbullet & \sqbullet \\ \sqbullet & \sqbullet \end{array}  & \begin{array}{ll} \sqbullet & \bullet \\ \bullet & \\ \bullet &  \end{array} & \begin{array}{l}\sqbullet \\ \bullet \\ \bullet \\ \bullet  \end{array} \\
\hspace{2mm} \left( \begin{array}{lll} 1 & 1 & 1 \\  & &  \end{array} \right)_1 \hspace{2mm} & \hspace{2mm} \left( \begin{array}{ll} 1 & 1  \\  1 &  \end{array} \right)_1 \hspace{2mm} & \hspace{2mm} \left( \begin{array}{l}  \\    \end{array} \right)_2 \hspace{2mm} & \hspace{2mm} \left( \begin{array}{ll} 1 & \\  1 & 1 \end{array} \right)_1 \hspace{2mm} & \hspace{2mm} \left( \begin{array}{lll} & & \\ 1 & 1 & 1  \end{array} \right)_1 \hspace{2mm}\\
\end{array}
\]

Andrews defined the {\em rank} of a Durfee symbol to be the length of the partition in the top row, minus the length of the partition in the bottom row.  Notice that this gives Dyson's original rank of the associated partition.  Andrews refined this idea by defining $n$-marked Durfee symbols, which use $n$ copies of the integers.   For example, the following is a $3$-marked Durfee symbol of $55$, where $\alpha^j,\beta^j$ indicate the partitions in their respective columns.
\[
\left(
\begin{array}{cc|ccc|c}
4_3 & 4_3 & 3_2 & 3_2 & 2_2 & 2_1 \\
       & 5_3 &        & 3_2 & 2_2 & 2_1
\end{array}
\right)_5
=:
\left(
\begin{array}{c|c|c}
\alpha^3 & \alpha^2 & \alpha^1 \\
\beta^3 & \beta^2 & \beta^1
\end{array}
\right)_5
\]
Each $n$-marked Durfee symbol has $n$ ranks, one defined for each column.  Let  $\rm{len}(\pi)$ denote the  length of a partition $\pi$.  Then the $n$th rank is defined to be  $\rm{len}(\alpha^n) - \rm{len}(\beta^n)$, and each $j$th rank for $1\leq j <n$ is defined by $\rm{len}(\alpha^j) - \rm{len}(\beta^j) -1$.  Thus the above example has $3$rd rank equal to $1$, $2$nd rank equal to $0$, and $1$st rank equal to $-1$. 

Let $\mathcal{D}_n(m_1,m_2,\dots, m_n;r)$ denote the number of $n$-marked Durfee symbols arising from partitions of $r$ with $i$th rank equal to $m_i$.  In \cite{Andrews},  Andrews  showed that the $( n+1)$-variable rank generating function for Durfee symbols may be expressed in terms of certain $q$-hypergeometric series, analogous to (\ref{rankgenfn}).  To describe this, for $n\geq 2$, define 
{\small{\begin{align}\label{rkorigdef}
	&R_n({\boldsymbol{x}};q) := \\ & \nonumber \mathop{\sum_{m_1 > 0}}_{m_2,\dots,m_n \geq 0} \!\!\!\!\!\!\!\! \frac{q^{(m_1 + m_2 + \dots + m_n)^2 + (m_1 + \dots + m_{n-1}) + (m_1 + \dots + m_{n-2}) + \dots + m_1}}{(x_1q;q)_{m_1} \!\left(\frac{q}{x_1};q\right)_{m_1} \!\!\!\!(x_2 q^{m_1};q)_{m_2 + 1} \!\!\left(\frac{q^{m_1}}{x_2};q\right)_{m_2+1} \!\!\!\!\!\!\!\!\!\!\cdots(x_n q^{m_1 + \dots + m_{n-1}};q)_{m_n+1} \!\!\left(\!\frac{q^{m_1 + \dots + m_{n-1}}}{x_n};q\!\right)_{\! m_n+1}},\end{align}}}where ${\boldsymbol{x}} = {\boldsymbol{x}}_n := (x_1,x_2,\dots,x_n).$ For $n=1$, the function $R_1(x;q)$ is defined as the $q$-hypergeometric series in (\ref{rankgenfn}).  
In what follows, for ease of notation, we may also write $R_1({\boldsymbol{x}};q)$ to denote $R_1(x;q)$, with the understanding that ${\boldsymbol{x}} := x$.  
 In \cite{Andrews}, Andrews  established   the following result, generalizing (\ref{rankgenfn}).  
\begin{theoremno}[\cite{Andrews}  Theorem 10]  For $n\geq 1$ we have that \begin{align}\label{durfgenand1}\sum_{m_1,m_2,\dots,m_n = -\infty}^\infty \sum_{r=0}^\infty \mathcal{D}_n(m_1,m_2,\dots,m_n;r)x_1^{m_1}x_2^{m_2}\cdots x_n^{m_n}q^r = R_n({\boldsymbol{x}};q).\end{align}
\end{theoremno} 
When $n=1$, one recovers Dyson's rank, that is, $\mathcal D_1(m_1;r)=N(m_1,r)$, so that \eqref{durfgenand1} reduces to \eqref{rankgenfn} in this case.  The mock modularity of the associated two variable generating function $R_1(x_1;q)$ was established in \cite{BO} as described in the Theorem above. When $n=2$, the modular properties of $R_2(1,1;q)$ were originally studied by Bringmann in \cite{Bri1}, who showed that 
\[R_2(1,1;q) := \frac{1}{(q;q)_\infty}\sum_{m\neq 0} \frac{(-1)^{m-1}q^{3m(m+1)/2}}{(1-q^m)^2}\]  
is a \emph{quasimock theta function}.   In \cite{BGM}, Bringmann, Garvan, and Mahlburg showed more generally that $R_{n}(1,1,\dots,1;q)$ is a quasimock theta function for $n\geq 2$. (See \cite{Bri1, BGM} for precise details of these statements.)  

In \cite{F-K}, two of the authors  established the automorphic properties of  $R_n\left({\boldsymbol{x}};q\right)$, for more arbitrary parameters ${\boldsymbol{x}} = (x_1,x_2,\dots,x_n)$, thereby treating families of $n$-marked Durfee rank functions with  additional singularities beyond those of $R_n(1,1,\dots,1;q)$.   We point out that the techniques of Andrews \cite{Andrews} and Bringmann \cite{Bri1} were not directly applicable in this setting  due to the presence of such additional singularities.   These singular combinatorial families are essentially mixed mock and quasimock modular forms. To precisely state a result from \cite{F-K} along these lines, we first introduce some notation, which we also use for the remainder of this paper. 
Namely, we consider functions evaluated at certain  length $n$ vectors ${\boldsymbol{\zeta_n}}$  of roots of unity defined as follows (as in \cite{F-K}).   

  In what follows, we let  $n$ be a fixed  integer satisfying $n\geq 2$.    Suppose for $1\leq j \leq n$, $\alpha_j \in \mathbb Z$ and $\beta_j \in  \mathbb N$, where $\beta_j \nmid \alpha_j, \beta_j \nmid 2\alpha_j$, and that $\frac{\alpha_{r}}{\beta_{r}} \pm \frac{\alpha_{s}}{\beta_{s}} \not\in\mathbb Z$ if $1\leq r\neq s \leq n$. Let
\begin{align} 
\notag
 {\boldsymbol{\alpha_n}}  &:= \Big( \frac{\alpha_{1}}{\beta_{1}},\frac{\alpha_{2}}{\beta_{2}},\dots,\frac{\alpha_{n}}{\beta_{n}} \Big) \in \mathbb Q^n \\ 
 \label{zetavec}
 {\boldsymbol{\zeta_n}}    &:=\big(\zeta_{\beta_{1}}^{\alpha_{1}},\zeta_{\beta_{2}}^{\alpha_{2}},\dots,\zeta_{\beta_{n}}^{\alpha_{n}}\big) \in \mathbb C^n.
 \end{align}  
 
 \begin{remark} We point out that the dependence of the vector $\boldsymbol{\zeta_n}$ on $n$ is reflected only in the length of the vector, and not (necessarily) in the roots of unity that comprise its components.  In particular, the vector  components may be chosen to be $m$-th roots of unity for different values of $m$.   
  \end{remark}  

\begin{remark}  The conditions stated above for $\boldsymbol{\zeta_n}$, as given in \cite{F-K},  do not require $\gcd(\alpha_j, \beta_j) = 1$.  Instead, they merely require that $\frac{\alpha_j}{\beta_j}   \not\in  \frac{1}{2}\Z$. Without loss of generality, we will assume here  that $\gcd(\alpha_j, \beta_j) = 1$.  Then, requiring that $\beta_j \nmid 2\alpha_j$ is the same as saying $\beta_j \neq 2$.    \end{remark}
  In \cite{F-K}, the authors proved that (under the hypotheses  for $\boldsymbol{\zeta_n}$   given above) the completed nonholomorphic function   
 \begin{equation}\label{Ahat}
 \widehat{\mathcal A}(\boldsymbol{\zeta_n};q) = q^{-\frac{1}{24}}R_n(\boldsymbol{\zeta_n};q) +  \mathcal A^-(\boldsymbol{\zeta_n};q)
 \end{equation} 
 transforms like a modular form. Here  the nonholomorphic part $\mathcal A^-$ is defined by 
 \begin{equation}\label{def_A-}
 \mathcal A^-(\boldsymbol{\zeta_n};q) := \frac{1}{\eta(\tau)}\sum_{j=1}^{n} (\zeta_{2\beta_j}^{-3\alpha_j}-\zeta_{2\beta_j}^{-\alpha_j})\frac{\mathscr{R}_3^-\left(\frac{\alpha_j}{\beta_j},-2\tau;\tau\right)}{\Pi_{j}^\dag({\boldsymbol{\alpha_n}})},
\end{equation}  
where $\mathscr{R}_3$ is defined in \eqref{AminusDef}, and  the constant $\Pi_j^{\dag}$ is  defined in \cite[(4.2), with  $n\mapsto j$ and $k\mapsto n$]{F-K}.  
Precisely, we have the following  special case of a theorem  established by two of the authors in \cite{F-K}.  
\begin{theoremno}[\cite{F-K} Theorem 1.1] If $n\geq 2$ is an integer,  then  $ \widehat{\mathcal A}\!\left( {\boldsymbol{\zeta_n}};q \right)$   is a nonholomorphic modular form of weight $1/2$ on $\Gamma_{n}$ with character $\chi_\gamma^{-1}$. 
\end{theoremno}   Here, the subgroup $\Gamma_{n}\subseteq \textnormal{SL}_2(\mathbb Z)$ under which $\widehat{\mathcal A}({\boldsymbol{\zeta_n}};q)$ transforms is defined by
\begin{align} \label{def_Gammangroup}
	\Gamma_{n}:=\bigcap_{j=1}^{n} \Gamma_0\left(2\beta_j^2\right)\cap \Gamma_1(2\beta_j),
\end{align} 	
and the Nebentypus character $\chi_\gamma$ is given in Lemma \ref{ETtrans}.  

\subsection{Quantum modular forms}   
In this paper, we study the quantum modular properties of the $(n+1)$-variable rank generating function for $n$-marked Durfee symbols $R_n({\boldsymbol{x}};q)$.  Loosely speaking, a quantum modular form is similar to a mock modular form in that it exhibits a modular-like transformation with respect to the action of a suitable subgroup of $\textnormal{SL}_2(\mathbb Z)$; however, the domain of a quantum modular form is not the upper half-plan $\mathbb H$, but rather the set of rationals $\mathbb Q$ or an appropriate subset.    The formal definition of a quantum modular form was originally introduced by Zagier in \cite{Zqmf} and has been slightly modified to allow for half-integral weights, subgroups of $\operatorname{SL_2}(\mathbb{Z})$, etc.\ (see \cite{BFOR}).

\begin{defn} \label{qmf}
A weight $k \in \frac{1}{2} \mathbb{Z}$ quantum modular form is a complex-valued function $f$ on $\mathbb{Q}$, such that for all $\gamma = \sm abcd \in \operatorname{SL_2}(\mathbb{Z})$, the functions $h_\gamma: \mathbb{Q} \setminus \gamma^{-1}(i\infty) \rightarrow \mathbb{C}$ defined by  
\begin{equation*}
h_\gamma(x) := f(x)-\varepsilon^{-1}(\gamma) (cx+d)^{-k} f\left(\frac{ax+b}{cx+d}\right)
\end{equation*}
satisfy a ``suitable" property of continuity or analyticity in a subset of $\mathbb{R}$.  
\end{defn}

\begin{remark} The complex numbers $\varepsilon(\gamma)$, which satisfy $|\varepsilon(\gamma)|=1$, are such as those appearing in the theory of half-integral weight modular forms. 
\end{remark}
\begin{remark}
We may modify Definition \ref{qmf} appropriately to allow transformations on subgroups of $\operatorname{SL_2}(\mathbb{Z})$. We may also restrict the domains of the functions $h_\gamma$ to be suitable subsets of $\mathbb{Q}$. 
\end{remark}

The subject of quantum modular forms has been widely studied since the time of origin of the above definition.   For example, quantum  modular forms have been shown to be related to the diverse areas of Maass forms, Eichler integrals, partial theta functions, colored Jones polynomials, meromorphic Jacobi forms, and vertex algebras, among other things (see \cite{BFOR} and references therein).  In particular, the notion of a quantum modular form is now known to have direct connection to Ramanujan's original definition of a mock theta function.  Namely, in his last letter to Hardy, Ramanujan examined the asymptotic difference between mock theta and modular theta functions as $q$ tends towards roots of unity $\zeta$ radially within the unit disk (equivalently, as $\tau$ approaches  rational numbers vertically in the upper half plane, with $q=e^{2\pi i \tau}, \tau \in \mathbb H$), and we now know that these radial limit differences are equal to  special values of quantum modular forms at rational numbers (see  \cite{BFOR, BR, FOR}).  
  
\subsection{Results} \label{sec_results}
On one hand, exploring the quantum modular properties of  the rank generating function for $n$-marked Durfee symbols $R_n$ in (\ref{durfgenand1})   is a natural problem given that two of the authors have established automorphic properties of this function on $\mathbb H$ (see  \cite[Theorem 1.1]{F-K} above),   that $\mathbb Q$ is a natural boundary to $\mathbb H$,  and that there has been much progress made in understanding the relationship between quantum modular forms and mock modular forms recently \cite{BFOR}.   Moreover, given that $R_n$ is a vast generalization of the two variable rank generating function in \eqref{rankgenfn} - both a combinatorial $q$-hypergeometric series and a mock modular form - understanding its automorphic properties in general is of interest.  
On the other hand, there is no reason to a priori expect $R_n$ to converge on $\mathbb Q$, let alone exhibit quantum modular properties there.  Nevertheless, we establish quantum modular properties for the rank generating function for $n$-marked Durfee symbols $R_n$ in this paper. 

For the remainder of this paper, we use the notation $$\mathcal V_{n}(\tau) := \mathcal V({\boldsymbol{\zeta_n}};q),$$ where $\mathcal V$ may refer to any one of the functions
 $\widehat{\mathcal A},  \mathcal A^-,  R_n,$ etc.   Moreover, we will write
 \begin{equation}\label{rel_AR}
 \mathcal A_{n}(\tau)    = q^{-\frac{1}{24}} R_n(\boldsymbol{\zeta_n};q)
  \end{equation} 
   for the holomorphic part of $\widehat{\mathcal A}$;  from \cite[Theorem 1.1]{F-K} above, we have that this function is a mock modular form of weight $1/2$ with character $\chi_\gamma^{-1}$  (see Lemma \ref{Chi_gammaForm})   for the group $\Gamma_n$ defined in (\ref{def_Gammangroup}).  Here, we will show that $\mathcal{A}_n$ is also a quantum modular form, under the  action of a subgroup  $\Gamma_{\boldsymbol{\zeta_n}} \subseteq \Gamma_{n}$ defined in \eqref{eqn:GroupDefn},   with quantum set 
\begin{equation}\label{qSetDef} \qs := \left\{\frac{h}{k}\in \Q\; \middle\vert\; \begin{aligned} & \ h\in\Z, k\in\N, \gcd(h,k) = 1, \ \beta_j \nmid k\ \forall\ 1\le j\le n,\\&\left\vert \frac{\alpha_j}{\beta_j}k - \left[\frac{\alpha_j}{\beta_j}k\right]\right\vert > \frac{1}{6}\ \forall\ 1\le j\le n\end{aligned} \right\},\end{equation}
where $[x]$ is the closest integer to $x$.

\begin{remark}\label{rmk:closest_int} For $x \in \frac12 + \mathbb Z$,  different sources define $[x]$ to mean
either $x-\frac12$ or $x+\frac12$. The definition of $\qs$ involving $[ \cdot ]$ is well-defined for either of these conventions in the case of $x\in \frac12 + \mathbb Z,$ as $\vert x - [x]\vert = \frac{1}{2}$.\end{remark}

 To define the exact subgroup under which $\mathcal A_n$ transforms as a quantum automorphic object, we let 
\begin{equation}\label{def_ell}\ell = \ell(\bs{\zeta_n}):= \begin{cases} 6\left[\text{lcm}(\beta_1, \dots, \beta_{n})\right]^2 &\text{ if } 3 \nmid \beta_j \text{ for all } 1\leq j \leq n, \\ 2\left[\text{lcm}(\beta_1, \dots, \beta_{n})\right]^2 &\text{ if } 3 \mid \beta_j \text{ for some } 1\leq j \leq n,  \end{cases}\end{equation} 
and let $S_\ell:=\left(\begin{smallmatrix}1 & 0 \\ \ell & 1 \end{smallmatrix}\right)$, $T:=\left(\begin{smallmatrix}1 & 1 \\ 0 & 1 \end{smallmatrix}\right)$.  We then define the group generated by these two matrices as 
\begin{equation} \label{eqn:GroupDefn}
\qSubgroup:= \langle S_\ell, T \rangle.
\end{equation}  

We now state our first main result,   which proves that $\mathcal A_n(x),$ and hence $e(-\frac{x}{24})R_n(\boldsymbol{\zeta_n};e(x))$   is a quantum modular form on $Q_{\boldsymbol{\zeta_n}}$ with respect to $\Gamma_{\boldsymbol{\zeta_n}}$. Here and throughout we let $e(x):=e^{2\pi ix}$.  
\begin{theorem}\label{thm_main_N0} 
 Let $n \geq 2$.   For all $\gamma = \left(\begin{smallmatrix}a & b \\ c & d \end{smallmatrix}\right) \in \qSubgroup$, and $x\in \qs$,  \[H_{n,\gamma}(x) := \mathcal{A}_n(x) - \chi_\gamma (c x+ d)^{-\frac12}\mathcal{A}_n(\gamma x) \] 
is defined, and extends to an analytic function in $x$ on $\mathbb{R} - \{\frac{-c}{d}\}$. 
In particular, for the matrix $S_\ell$,  
\begin{multline}\notag
H_{n,S_\ell}(x) = \frac{\sqrt{3}}{2} \sum_{j=1}^{n}\frac{(\zeta_{2\beta_j}^{\alpha_j} - \zeta_{2\beta_j}^{3\alpha_j})}{\displaystyle\Pi^\dag_j( {\bs{\alpha_n}})}  \left[\sum_\pm \zeta_6^{\pm1}\int_{\frac{1}{\ell}}^{i\infty}\frac{g_{\pm\frac13+\frac12,-\frac{3\alpha_j}{\beta_j}+\frac12}(3\rho)}{\sqrt{-i(\rho+x)}}d\rho \right] \\
+\sum_{j=1}^{n}\frac{(\zeta_{2\beta_j}^{-3\alpha_j} - \zeta_{2\beta_j}^{-\alpha_j})}{\displaystyle\Pi^\dag_j( {\bs{\alpha_n}})} (\ell x+1)^{-\frac12}\zeta_{24}^{-\ell}\mathcal{E}_1\left(\frac{\alpha_j}{\beta_j},\ell;x\right),
\end{multline}
where  the weight $3/2$ theta functions $g_{a,b}$ are defined in \eqref{def_gab}, and $\mathcal E_1$ is defined in \eqref{def_mathcalE}. 
\end{theorem}

 \begin{remark} As mentioned above, the   constants  $\Pi^\dagger_j$ are defined in \cite[(4.2)]{F-K}.  With the exception of replacing $n\mapsto j$ and $k\mapsto n$, we have preserved the notation for these constants from \cite{F-K}.   
\end{remark}  

 \begin{remark}
Our results apply to any $n\geq 2$, as the quantum modular properties in the case $n=1$ readily follow from existing results.  Namely, proceeding as in the proof of Theorem \ref{qsProof}, one may determine a suitable  quantum set for the normalized rank generating function in \cite[Theorem 1.1]{BO}.   Using \cite[Theorem 1.1]{BO}, a short calculation shows that the error to modularity (with respect to the nontrivial generator of $\Gamma_c$)   is a multiple of $$\int_{x}^{i\infty} \frac{\Theta(\frac{a}{c};\ell_c \rho)}{\sqrt{-i(\tau+\rho)}}d\rho$$ for some $x\in\mathbb Q$.  When viewed as a function of $\tau$ in a subset of $\mathbb R$,  this integral is analytic (e.g., see \cite{LZ, Zqmf}).  

One could also establish the quantum properties of  a non-normalized version of   $R_1$ by  rewriting it in terms of the Appell-Lerch sum $A_3$, and proceeding as in the proof of Theorem \ref{thm_main_N0}.  In this case, 
 $R_1(\zeta_1;q)$ (where $\zeta_1=e(\alpha_1/\beta_1)$) converges on the quantum set $Q_{\zeta_1}$, where this set is defined by  letting $n=1$ in \eqref{qSetDef}.  
 
The interested reader may also wish to consult \cite{CLR} for general results on quantum properties associated to mock modular forms.  
 \end{remark}  

\begin{remark}
In a forthcoming joint work \cite{FJKS}, we extend Theorem \ref{thm_main_N0} to hold for the more general vectors of roots of unity considered in \cite{F-K}, i.e., those with repeated entries.  Allowing repeated roots of unity introduces additional singularities, and the modular completion of $R_n$ is significantly more complicated.  This precludes us from proving the more general case in the same way as the restricted case we address here.
\end{remark}

\section{Preliminaries}\label{prelim}  
\subsection{Modular, mock modular and Jacobi forms}
A  special ordinary modular form we require   is Dedekind's $\eta$-function, defined in (\ref{def_eta}).  This function  is well known to satisfy the following transformation law \cite{Rad}.  \begin{lemma}\label{ETtrans}
 For
$\gamma=\sm{a}{b}{c}{d} \in \textnormal{SL}_2(\mathbb Z)$, we have that
\begin{align*}
\eta\left(\gamma\tau\right)  = \chi_\gamma(c\tau + d)^{\frac{1}{2}} \eta(\tau), 
\end{align*}
where
$$\chi_\gamma := \begin{cases} e\left(\frac{b}{24}\right), & \textnormal{ if } c=0, d=1,  \\
\sqrt{-i} \ \omega_{d,c}^{-1}e\left(\frac{a+d}{24c}\right), & \textnormal{ if } c>0,\end{cases}$$ with $\omega_{d,c} := e(\frac12 s(d,c))$. Here the Dedekind sum $s(m,t)$ is given for integers $m$ and $t$    
by
$$s(m,t) := \sum_{j \!\!\!\mod t} \left(\!\!\left(\frac{j}{t}\right)\!\!\right)\left(\!\!\left(\frac{mj}{t}\right)\!\!\right),$$ 
where $((x)) := x-\lfloor x \rfloor - 1/2$ if $x\in \mathbb R\setminus \mathbb Z$, and $((x)):=0$ if $x\in \mathbb Z$.
 \end{lemma}

The following gives a useful expression for $\chi_\gamma$ (see \cite[Ch. 4, Thm. 2]{Knopp}):
\begin{equation}\label{Chi_gammaForm}
\chi_\gamma =
\left\{ \begin{array}{ll}
\big(\frac{d}{|c|} \big)e\left(\frac{1}{24}\left( (a+d)c - bd(c^2-1) - 3c \right)\right)
& \mbox{ if } c \equiv 1 \pmod{2}, \\				
\big( \frac{c}{d} \big) e\left(\frac{1}{24}\left( (a+d)c - bd(c^2-1) + 3d - 3 - 3cd \right)\right)
& \mbox{ if } d\equiv 1\pmod{2},
\end{array}\right.
\end{equation}
where $\big(\frac{\alpha}{\beta}\big)$ is the generalized Legendre symbol.
 
We require two additional  functions, namely the Jacobi theta function $\vartheta(u;\tau)$, an ordinary Jacobi form, and a nonholomorphic  modular-like function $R(u;\tau)$ used by Zwegers in \cite{Zwegers1}.    In what follows, we will also need certain transformation properties of these functions.

\begin{proposition} \label{thetaTransform} For $u \in\C$ and $\tau\in\mathbb{H}$, define
\begin{equation}\label{thetaDef}\vartheta(u;\tau) := \sum_{\nu\in\frac{1}{2} + \Z} e^{\pi i \nu^2\tau + 2\pi i \nu\left(u + \frac{1}{2}\right)}.\end{equation}
Then $\vartheta$ satisfies
\begin{enumerate}
	\item $\vartheta(u+1; \tau) = -\vartheta(u; \tau),$\\
	\item $\vartheta(u + \tau; \tau) = -e^{-\pi i \tau - 2\pi i u}\vartheta(u; \tau),$\\
	\item $\displaystyle \vartheta(u; \tau) = - i e^{\pi i \tau/4}e^{-\pi i u} \prod_{m=1}^\infty (1-e^{2\pi i m\tau})(1-e^{2\pi i u}e^{2\pi i \tau(m-1)})(1 - e^{-2\pi i u}e^{2\pi i m\tau}).$
\end{enumerate}
\end{proposition}  
The nonholomorphic function $R(u;\tau)$   is defined in \cite{Zwegers1} by
\[
R(u;\tau):=\sum_{\nu\in\frac12+\Z} \left\{\operatorname{sgn}(\nu)-E\left(\left(\nu+\frac{\operatorname{Im}(u)}{\operatorname{Im}(\tau)}\right)\sqrt{2\operatorname{Im}(\tau)}\right)\right\}(-1)^{\nu-\frac12}e^{-\pi i\nu^2\tau-2\pi i\nu u},
\]
where
\[
E(z):=2\int_0^ze^{-\pi t^2}dt.
\]
 
 The function $R$ transforms like a (nonholomorphic) mock Jacobi form as follows.  
\begin{proposition}[Propositions 1.9 and 1.10, \cite{Zwegers1}]\label{Rtransform} The function $R$ satsifies the following   transformation properties:
\begin{enumerate}
	\item $R(u+1;\tau) = -R(u; \tau),$\\
	\item $R(u; \tau) + e^{-2\pi i u - \pi i\tau}R(u+\tau; \tau) = 2e^{-\pi i u - \pi i \tau/4}$,\\
	\item  $R(u;\tau) = R(-u;\tau)$,   \\ 
        \item $R(u;\tau+1)=e^{-\frac{\pi i}{4}} R(u;\tau)$,\\
        \item $\frac{1}{\sqrt{-i\tau}} e^{\pi i u^2/\tau} R\left(\frac{u}{\tau};-\frac{1}{\tau}\right)+R(u;\tau)=h(u;\tau),$
where the Mordell integral is defined by \begin{align}\label{def_hmordell} h(u;\tau):=\int_\R \frac{e^{\pi i\tau t^2-2\pi ut}}{\cosh \pi t} dt. \end{align} 
 
\end{enumerate}
\end{proposition}

 Using the functions $\vartheta$ and $R$, Zwegers 
 defined   the 
nonholomorphic function  
 \begin{align}\label{AminusDef}
	\mathscr R_3 (u, v;\tau) :=& \frac{i}{2} \sum_{j=0}^{2} e^{2\pi i j u} \vartheta(v + j\tau + 1; 3\tau) R(3u - v - j\tau - 1; 3\tau)\\
	=& \frac{i}{2} \sum_{j=0}^{2} e^{2\pi i j u} \vartheta(v + j\tau; 3\tau) R(3u - v - j\tau; 3\tau),\nonumber
\end{align}
where the equality of the two expressions in \eqref{AminusDef} is justified by Proposition \ref{thetaTransform} and Proposition \ref{Rtransform}.   
 This function is used to complete the level three Appell function  (see \cite{Zwegers2} or \cite{BFOR}) 
\begin{align*} A_3(u, v; \tau) := e^{3\pi i u} \sum_{n\in\Z} \frac{(-1)^n q^{3n(n+1)/2}e^{2\pi i nv}}{1 - e^{2\pi i u}q^n},\end{align*} where $u,v \in \mathbb C$,
as  
\begin{align*}\label{def_A3hat} \widehat{A}_3(u, v; \tau) := A_3(u, v;\tau) + \mathscr R_3(u, v;\tau). \end{align*}

This completed  function transforms like a (non-holmorphic) Jacobi form, and in particular satisfies  the following elliptic transformation.  

\begin{theorem}[{\cite[Theorem 2.2]{Zwegers2}}]\label{completeAtransform} For $n_1, n_2, m_1, m_2\in\Z$, the completed level $3$ Appell function $\widehat{A}_3$ satisfies
\[\widehat{A}_3(u + n_1\tau + m_1, v + n_2\tau + m_2; \tau) = (-1)^{n_1 + m_1}e^{2\pi i (u(3n_1 - n_2) - vn_1)}q^{3n_1^2/2 - n_1n_2}\widehat{A}_3(u, v; \tau).\]
\end{theorem}
 
 The following relationship between the Appell series $A_3$ and the combinatorial series $R_n$ is proved in \cite{F-K}.
 \begin{props}[{\cite[Proposition 4.2]{F-K}}]  Under the hypotheses given above on $\boldsymbol{\zeta_n}$, we have that 
\[
  R_n(\boldsymbol{\zeta_n};q) =\frac{1}{(q)_\infty} \sum_{j=1}^n\left(\zeta_{2\beta_j}^{-3\alpha_j}-\zeta_{2\beta_j}^{-\alpha_j}\right)\frac{A_3\left(\frac{\alpha_j}{\beta_j},-2\tau;\tau\right)}{\Pi^\dag_j( {\boldsymbol{\alpha_n}})}.
\]
\end{props} 

We also note that
\begin{align*}
\widehat{\mathcal A}_n\left( \tau \right)  = \frac{1}{\eta(\tau)}\sum_{j=1}^{n} (\zeta_{2\beta_j}^{-3\alpha_j}-\zeta_{2\beta_j}^{-\alpha_j})\frac{\widehat{A}_3\left(\frac{\alpha_j}{\beta_j},-2\tau;\tau\right)}{\Pi_{j}^\dag({\boldsymbol{\alpha_n}})}. 
\end{align*} 
 
 In addition to working with the Appell sum $\widehat{A}_3$, we also make use of additional properties  of the functions $h$ and $R$.  In particular,  Zwegers also  showed how under certain hypotheses, these functions  can be written in terms of integrals involving the weight $3/2$ modular forms $g_{a,b}(\tau)$, defined for $a,b\in\mathbb R$ and $\tau \in \mathbb H$ by
 \begin{align}\label{def_gab}
 g_{a,b}(\tau) := \sum_{\nu \in a + \mathbb Z} \nu e^{\pi i \nu^2\tau + 2\pi i \nu b}.
 \end{align}
 We will make use of the following results. 
\begin{proposition}[{\cite[Proposition 1.15 (1), (2), (4), (5)]{Zwegers1}}]\label{prop_Zg}
 The function  $g_{a,b}$ satisfies:
\begin{enumerate}
\item[(1)] $g_{a+1,b}(\tau)= g_{a,b}(\tau)$, \\
\item[(2)] $g_{a,b+1}(\tau)= e^{2\pi ia} g_{a,b}(\tau)$,\\
\item[(3)] $g_{a,b}(\tau+1)= e^{-\pi ia(a+1)}g_{a,a+b+\frac12}(\tau)$,\\
\item[(4)] $g_{a,b}(-\frac{1}{\tau})= i e^{2\pi iab} (-i\tau)^{\frac32}g_{b,-a}(\tau)$.
\end{enumerate}
\end{proposition}
 
  \begin{theorem}[{\cite[Theorem 1.16 (2)]{Zwegers1}}]\label{thm_Zh2}
 Let $\tau \in \mathbb H$. For $a,b \in (-\frac12,\frac12)$, we have
 $$h(a\tau-b;\tau) = -  e\left(\tfrac{a^2\tau}{2} - a(b+\tfrac12)\right) \int_{0}^{i\infty} \frac{g_{a+\frac12,b+\frac12}(\rho)}{\sqrt{-i(\rho+\tau)}}d\rho.$$ 
 \end{theorem}

\section{The quantum set}\label{quantumSet}


 We call a subset $S \subseteq \Q$ a {\em quantum set} for a function $F$ with respect to the group $G\subseteq \textnormal{SL}_2(\Z)$ if both $F(x)$ and $F(Mx)$ exist (are non-singular) for all $x\in S$ and $M\in G$.   

In this section, we will show that $\qs$ as defined in \eqref{qSetDef} is a quantum set for $\mathcal{A}_n$ with respect to the group $\qSubgroup$.  Recall that $\qs$ is defined as

\begin{align*} \qs := \left\{\frac{h}{k}\in \Q\; \middle\vert\; \begin{aligned} & \ h\in\Z, k\in\N, \gcd(h,k) = 1, \ \beta_j \nmid k\ \forall\ 1\le j\le n,\\&\left\vert \frac{\alpha_j}{\beta_j}k - \left[\frac{\alpha_j}{\beta_j}k\right]\right\vert > \frac{1}{6}\ \forall\ 1\le j\le n\end{aligned} \right\},\end{align*}
where $[x]$ is the closest integer to $x$ (see Remark \ref{rmk:closest_int}).


Moreover, recall that the ``holomorphic part''  we consider (see \S \ref{sec_results}) is  $\mathcal A_n(\tau) = q^{-\frac{1}{24}} R_n(\boldsymbol{\zeta_n}; q)$.  To show that $\qs$ is a quantum set for $\mathcal A_n(\tau)$, we must first show that the the multi-sum defining $R_n(\boldsymbol{\zeta_n}; \zeta_k^h)$ converges for $\frac{h}{k}\in\qs$.  In what follows, as in the definition of $\qs$, we take $h\in\Z$, $k\in\N$ such that $\gcd(h,k) = 1$.

We start by addressing the restriction that for $\frac{h}{k}\in \qs$, $\beta_j\nmid k$ for all $1 \le j \le n$.

\begin{lemma} For $\frac{h}{k}\in \Q$, all summands of $R_n(\boldsymbol{\zeta_n}; \zeta_k^h)$ are finite if and only if $\beta_j \nmid k$ for all $1\le j \le n$.\end{lemma}

\begin{proof}
Examining the multi-sum $R_n(\boldsymbol{\zeta_n}; \zeta_k^h)$, we see that all terms are a power of $\zeta_k^h$ divided by a product of factors of the form $1 - \zeta_{\beta_j}^{\pm\alpha_j} \zeta_k^{hm}$ for some integer $m\ge 1$.  Therefore, to have each summand be finite, it is enough to ensure that $1 - \zeta_{\beta_j}^{\pm\alpha_j} \zeta_k^{hm} \neq 0$ for all $m\ge 1$ and for all $1\le j \le n$.  For ease of notation in this proof, we will omit the subscripts for $\alpha_j$ and $\beta_j$.

If $1 - \zeta_{\beta}^{\pm\alpha} \zeta_k^{hm} = 0$ for some $m\in\N$, we have that
\[\pm\frac{\alpha}{\beta} + \frac{hm}{k} \in\Z.\]
Let $K = \lcm(\beta, k) = \beta\beta^\prime = kk^\prime$. 
Then $\pm\frac{\alpha}{\beta} + \frac{hm}{k} \not\in\Z$ is the same as $\pm\alpha\beta^\prime + hmk^\prime \not\in K\Z$.  Since $K = kk^\prime$, if $k^\prime \nmid \alpha\beta^\prime$, this is always true and we do not have a singularity.

However, since $K = \beta\beta^\prime = kk^\prime$, if $k^\prime \vert \alpha\beta^\prime$, then $\frac{\beta\beta^\prime}{k} \vert \alpha\beta^\prime$.  This implies that $\beta \vert \alpha k$ and that $\beta \vert k$ since $\gcd(\alpha, \beta) = 1$.

Therefore, if $\beta\nmid k$, it is always the case that $k^\prime\nmid \alpha\beta^\prime$, so for all $m\in\N$, 
\[\pm\frac{\alpha}{\beta} + \frac{hm}{k} \not\in\Z.\]
\end{proof}

Now that we have shown that all summands in $R_n(\boldsymbol{\zeta_n}; \zeta_k^h)$ are finite for $\frac{h}{k}\in\qs$, we will show that the sum converges.

\begin{theorem}\label{qsProof} For $\boldsymbol{\zeta_n}$ as in \eqref{zetavec}, if $\frac{h}{k} \in \qs$, then $R_n(\boldsymbol{\zeta_n}; \zeta_k^h)$ converges and can be evaluated as a finite sum.  In particular,  we have that: 
 
\begin{multline} \label{eqn_Rnconvsum}
R_n(\boldsymbol{\zeta_n}; \zeta_k^h) = \prod_{j=1}^n \frac{1}{1 - ((1-x_j^k)(1-x_j^{-k}))^{-1}}\\
\times \!\!\!\!\! \sum_{\substack{0 < m_1\le k\\ 0 \le m_2, \dots, m_n < k}} \frac{\zeta_k^{h[(m_1 + m_2 + \dots + m_n)^2 + (m_1 + \dots + m_{n-1}) + (m_1 + \dots + m_{n-2}) + \dots + m_1]}}{(x_1\zeta_k^h;\zeta_k^h)_{m_1} \left(\frac{\zeta_k^h}{x_1};\zeta_k^h\right)_{m_1} (x_2 \zeta_k^{hm_1};\zeta_k^h)_{m_2 + 1} \left(\frac{\zeta_k^{hm_1}}{x_2};\zeta_k^h\right)_{m_2+1}} \\ 
\times \frac{1}{(x_3 \zeta_k^{h(m_1 + m_2)};\zeta_k^h)_{m_3 + 1}\!\!\left(\frac{\zeta_k^{h(m_1 + m_2)}}{x_3};\zeta_k^h\!\right)_{\!m_3 + 1} \!\!\!\!\!\!\!\!\!\!\cdots(x_n \zeta_k^{h(m_1 + \dots + m_{n-1})};\zeta_k^h)_{ m_n+1} \!\!\left(\!\frac{\zeta_k^{h(m_1 + \dots + m_{n-1})}}{x_n};\zeta_k^h\!\right)_{\!  m_n+1} },
\end{multline}   
where $\boldsymbol{\zeta_n} = (x_1, x_2, \dots, x_n)$.
\end{theorem}

\begin{proof}[Proof of Theorem \ref{qsProof}]
We start by taking $\frac{h}{k} \in \qs$, and write $\zeta = \zeta_k^h$.  For ease of notation, we will use $x_j$ to denote the $j$-th component in $\boldsymbol{\zeta_n}$, so $x_j = e^{2\pi i \alpha_j/\beta_j}$.  Further, for clarity of argument, we will carry out the proof in the case of $n = 2$, with comments throughout about how the proof follows for $n > 2$.  We have that 
\begin{align} \nonumber
R_2((x_1,x_2); \zeta) =& \sum_{\substack{m_1 > 0\\ m_2\ge 0}} \frac{\zeta^{(m_1+m_2)^2 + m_1}}{(x_1\zeta;\zeta)_{m_1}(x_1^{-1}\zeta;\zeta)_{m_1}(x_2\zeta^{m_1};\zeta)_{m_2+1}(x_2^{-1}\zeta^{m_1};\zeta)_{m_2+1}}\\
\label{sumRearranged3} =&\sum_{M_1, M_2 \ge 0} \frac{1}{(1 - x_1^k)^{M_1} (1 - x_1^{-k})^{M_1} (1-x_2^k)^{M_2}(1-x_2^{-k})^{M_2}}\\ \label{sumRearranged2}
	&\times \sum_{\substack{0 < s_1 \le k\\ 0\le s_2 < k}}\frac{\zeta^{(s_1+s_2)^2+s_1}}{(x_1\zeta;\zeta)_{s_1}(x_1^{-1}\zeta;\zeta)_{s_1}(x_2\zeta^{s_1};\zeta)_{s_2+1}(x_2^{-1}\zeta^{s_1};\zeta)_{s_2+1}},
	\end{align}
where we have let $m_j = s_j + M_j k$ for $0 < s_1 \le k$, $0 \le s_2 < k$, and $M_j\in\N_0$, and have used the fact that 
\[ (x\zeta^r;\zeta)_{s+Mk}  =   (1 - x^k)^M \; (x\zeta^r; \zeta)_s ,\] which holds for any $M, r, s\in\N_0$.  
(We note that  for $n > 2$, we proceed as above, additionally taking $0 \le s_j \le k-1$ for $j > 2$.)  
The second sum in \eqref{sumRearranged2} is a finite sum, as desired.  For the first sum in (\ref{sumRearranged3})  we notice that we in fact have the product of two geometric series, each of the form
\[\sum_{M_j\ge 0} \left(\frac{1}{ (1-x_j^k)( 1 - x_j^{-k})}\right)^{M_j}.\]
By definition, we have $x_j = \cos\theta_j + i\sin\theta_j$ where $\theta_j = \frac{2\pi\alpha_j}{\beta_j}$.  Therefore, this sum converges if and only if  \begin{align*}
	\vert 1 - x_j^k\vert  \vert 1 - x_j^{-k}\vert   =2 - 2\cos(k\theta_j) > 1 \iff \cos(k\theta_j) < \frac{1}{2}.
\end{align*} 
For $\cos(k\theta_j) < \frac{1}{2}$, it must be that $k\theta_j = r + 2\pi M$ where $-\pi < r \le \pi$, $\vert r \vert > \frac{\pi}{6}$, and $M \in\Z$.  This is equivalent to saying 
\[\left\vert \frac{\alpha_j}{\beta_j}k - \left[\frac{\alpha_j}{\beta_j}k\right]\right\vert > \frac{1}{6}\ \  \forall\ 1\le j\le n,\]
as in the definition of $\qs$ in \eqref{qSetDef}. Therefore, we see that for $\frac{h}{k}\in\qs$, $R_2((x_1, x_2); \zeta)$ converges to the claimed expression in \eqref{eqn_Rnconvsum}. 
 
We note that by Abel's theorem, having shown convergence of $R_2((x_1, x_2); \zeta)$, we have that $R_2((x_1, x_2); q)$ converges to $R_2((x_1, x_2); \zeta)$ as $q\to\zeta$ radially within the unit disc.

As noted, the above argument extends to $n > 2$.  Letting $m_j = s_j + M_j k$ with $0 < s_1 \le k$ and $0 \le s_j < k$ for $j \ge 2$, rewriting as in \eqref{eqn_Rnconvsum}, and then summing the resulting geometric series gives the desired exact formula for $R_n(\boldsymbol{\zeta_n}; \zeta)$.
\end{proof}

 To complete  the argument that $\qs$ is a quantum set for $R_n(\boldsymbol{\zeta_n}; \zeta)$ with respect to $\Gamma_{\boldsymbol{\zeta_n}}$, it remains to be seen that
 $R_n(\boldsymbol{\zeta_n}; \xi)$ converges, where $\xi = e^{2\pi i \gamma(\frac{h}{k})}$ for $\frac{h}{k}\in\qs$ and $\gamma\in \qSubgroup$, defined in \eqref{eqn:GroupDefn}.  For the ease of the reader, we recall  from \eqref{def_ell} and \eqref{eqn:GroupDefn}  that
\begin{align*} \qSubgroup := \left\langle \left(\begin{matrix} 1 & 1\\ 0 & 1\end{matrix}\right), \left(\begin{matrix} 1 & 0\\ \ell & 1\end{matrix}\right)\right\rangle,\end{align*}
where
\[\ell = \ell_{\beta} := \begin{cases} 6\left[\text{lcm}(\beta_1, \dots, \beta_{k})\right]^2 &\text{ if $\forall j$, $3\not\vert \beta_j$}\\ 2\left[\text{lcm}(\beta_1, \dots, \beta_{k})\right]^2 &\text{ if $\exists j$, $3 \vert \beta_j$.}\end{cases}\]
The convergence of $R_n(\boldsymbol{\zeta_n}; \xi)$ is a direct consequence of the following lemma.

\begin{lemma}\label{setClosed} The set $\qs$ is closed under the action of $\qSubgroup$.\end{lemma}

\begin{proof}
Since $\qSubgroup$ is given as a set with two generators, it is enough to show that $\qs$ is closed under action of each of those generators. 

Let $\frac{h}{k}\in\qs$.  Then $\sm{1}{1}{0}{1}\frac{h}{k} = \frac{h + k}{k}$.  Note that $\gcd(h+k, k) = \gcd(h,k) = 1$ and we already know that $k$ satisfies the conditions in the definition of $\qs$.  Therefore, $\sm{1}{1}{0}{1}\frac{h}{k}\in\qs$.

Under the action of $\sm{1}{0}{\ell}{1}$, we have
\[\left(\begin{array}{cc}1 & 0 \\ \ell & 1\end{array}\right)\frac{h}{k} = \frac{h}{h\ell + k}.\]
We first note that $\gcd(h, h\ell + k) = \gcd(h,k) = 1$, and $\beta_j \nmid (h\ell + k)$ as $\beta_j \vert \ell$ and $\beta_j \nmid k$.  It remains to check that
\[\left\vert \frac{\alpha_j}{\beta_j}(h\ell + k) - \left[\frac{\alpha_j}{\beta_j}(h\ell + k)\right]\right\vert > \frac{1}{6}\ \forall\ 1\le j\le n.\]
We have that
\begin{align}\nonumber
	\left\vert \frac{\alpha_j}{\beta_j} (h \ell + k) - \left[\frac{\alpha_j}{\beta_j}(h\ell + k)\right]\right\vert &= \left\vert \frac{\alpha_j h\ell}{\beta_j} + \frac{\alpha_j}{\beta_j}k - \left[\frac{\alpha_j h\ell}{\beta_j} + \frac{\alpha_j}{\beta_j}k\right]\right\vert\\
	&= \left\vert\frac{\alpha_j}{\beta_j} k - \left[\frac{\alpha_j}{\beta_j} k\right]\right\vert > \frac{1}{6},\label{closestIntSimplification}
\end{align}
where we can simplify as in \eqref{closestIntSimplification} since, by definition of $\ell$, $\frac{\alpha_j\ell}{\beta_j} \in\Z$.  Thus, $\qs$ is closed under the action of $\qSubgroup$. \end{proof}

\section{Proof of Theorem \ref{thm_main_N0}}\label{n0proof} 

We  now  prove Theorem \ref{thm_main_N0}. 
Our first goal is to establish that $H_{n,\gamma}$ is analytic in $x$ on $\mathbb{R} - \{\frac{-c}{d}\}$ for all $x\in \qs$ and $\gamma = \left(\begin{smallmatrix}a & b \\ c & d \end{smallmatrix}\right) \in \qSubgroup$.  As shown in Section \ref{quantumSet}, we have that $\mathcal A_n(x)$ and $\mathcal A_n(\gamma x)$ are defined for all $x\in\qs$ and $\gamma\in \qSubgroup$.  Note that it suffices to consider only the generators $S_\ell$ and $T$ of $\qSubgroup$, since $$H_{n,\gamma \gamma'}(\tau)= H_{n,\gamma'}(\tau) + \chi_{\gamma'}(C\tau+D)^{-\frac12} H_{n,\gamma}(\gamma'\tau)$$ for $\gamma = \left(\begin{smallmatrix}a & b \\ c & d \end{smallmatrix}\right)$ and $\gamma' = \left(\begin{smallmatrix}A & B \\ C & D \end{smallmatrix}\right)$.

First, consider $\gamma = T$.  Then by definition, $\chi_T=\zeta_{24}$, and so $H_{n,T}(x) = \mathcal A_n(x) - \zeta_{24} \mathcal A_n(x+1)$.
When we map $x \mapsto x +1$, $q=e^{2\pi i x}$ remains invariant.   Then since the definition of $R_{n}(x)$ in \eqref{rkorigdef} can be expressed as a series only involving integer powers of $q$, it is also invariant.  Thus 
$$\mathcal A_n(x+1)=e^{\frac{ -2\pi i   (x+1)}{24}}R_{n}(x) = \zeta_{24}^{-1}\mathcal A_n(x),$$ and so
$H_{n,T}(x)=0$.

We now consider the case $\gamma = S_\ell$.  In this case using \eqref{Chi_gammaForm} we calculate that $\chi_{S_\ell}=\zeta_{24}^{-\ell}$.  Thus, $$H_{n,S_\ell}(x) = \mathcal A_n(x) - \zeta_{24}^{-\ell}(\ell x +1)^{-\frac12} \mathcal A_n(S_\ell x).$$  From the modularity of  $\widehat{\mathcal A}_n$ we have that 
$\widehat{\mathcal A}_n(x) = \zeta_{24}^{-\ell}(\ell x +1)^{-\frac12}\widehat{\mathcal A}_n(S_\ell x)$. Thus \eqref{Ahat} and \eqref{rel_AR} give that 
\begin{equation}\label{eq:HviaA-}
H_{n,S_\ell}(x) =-\mathcal A_n ^-(x)+\zeta_{24}^{-\ell}(\ell x +1)^{-\frac12}\mathcal A_n^-(S_\ell x),
\end{equation}
 where $\mathcal A_n^-$ is defined in \eqref{def_A-}.  

Using the Jacobi triple product identity from Proposition \ref{thetaTransform} item (3), we can simplify the theta functions to get that  $\vartheta\left(-2\tau ;3\tau\right) = iq^{-\frac23}\eta(\tau)$, $\vartheta\left(-\tau ;3\tau\right) = iq^{-\frac16}\eta(\tau)$, and $\vartheta\left(0;3\tau\right) = 0$.  Thus, 
\begin{align*} 
\mathscr{R}_3\left( \frac{\alpha_j}{\beta_j},-2\tau ; \tau \right) = - \frac12 q^{-\frac23} \eta(\tau) \sum_{\delta=0}^1 e\left(\frac{\alpha_j}{\beta_j} \delta \right) q^{\frac{\delta}{2}} R\left(\frac{3\alpha_j}{\beta_j} + (2-\delta)\tau ; 3\tau \right). 
\end{align*}
Using Proposition \ref{Rtransform} item (2), we can rewrite 
$$R\left(\frac{3\alpha_j}{\beta_j} + 2\tau ; 3\tau \right) = 2e\left(\frac{3\alpha_j}{2\beta_j} \right) q^{\frac58} - e \left(\frac{3\alpha_j}{\beta_j} \right) q^{\frac12} R\left(\frac{3\alpha_j}{\beta_j} - \tau; 3\tau \right),$$ 
so that 
\begin{multline}\notag
\sum_{\delta=0}^1 e\left(\frac{\alpha_j}{\beta_j} \delta \right) q^{\frac{\delta}{2}} R\left(\frac{3\alpha_j}{\beta_j} + (2-\delta)\tau ; 3\tau \right) = \\ 
2e\left(\frac{3\alpha_j}{2\beta_j} \right) q^{\frac58} + e \left(\frac{2\alpha_j}{\beta_j} \right) q^{\frac12} \sum_{\pm} \pm e \left(\mp \frac{\alpha_j}{\beta_j} \right)R\left(\frac{3\alpha_j}{\beta_j} \pm \tau; 3\tau \right).
\end{multline} 
Thus we see that 
\begin{multline}\label{eq:F-}
\mathcal A_n^-(\tau) = -\frac12 \sum_{j=1}^{n}\frac{(\zeta_{2\beta_j}^{-3\alpha_j} - \zeta_{2\beta_j}^{-\alpha_j})}{\displaystyle\Pi^\dag_{j} ( {\bs{\alpha_k}})} e \left(\frac{2\alpha_j}{\beta_j} \right) q^{-\frac16} \sum_{\pm} \pm e \left(\mp \frac{\alpha_j}{\beta_j} \right) R\left(\frac{3\alpha_j}{\beta_j} \pm \tau; 3\tau \right) \\
-q^{-\frac{1}{24}} \sum_{j=1}^{n}\frac{(\zeta_{2\beta_j}^{-3\alpha_j} - \zeta_{2\beta_j}^{-\alpha_j})}{\displaystyle\Pi^\dag_{j} ( {\bs{\alpha_k}})} e \left(\frac{3\alpha_j}{2\beta_j} \right).
\end{multline}
Now to compute $\mathcal A_n^-(S_\ell \tau)$ we first define 
\begin{align*}
F_{\alpha,\beta}(\tau):= q^{-\frac16} \sum_{\pm}  \pm e \left(\mp \frac{\alpha}{\beta} \right) R\left(\frac{3\alpha}{\beta} \pm \tau; 3\tau \right).
\end{align*}
Then by \eqref{eq:HviaA-} and \eqref{eq:F-} we can write
\begin{multline*}
H_{n,S_\ell}(\tau) = \frac12 \sum_{j=1}^{n}\frac{(\zeta_{2\beta_j}^{-3\alpha_j} - \zeta_{2\beta_j}^{-\alpha_j})}{\displaystyle\Pi^\dag_{j} ( {\bs{\alpha_k}})} e \left(\frac{2\alpha_j}{\beta_j} \right) \left[ F_{\alpha_j,\beta_j}(\tau) - \zeta_{24}^{-\ell}(\ell \tau +1)^{-\frac12}F_{\alpha_j,\beta_j}(S_\ell \tau) \right] \\
 +\sum_{j=1}^{n}\frac{(\zeta_{2\beta_j}^{-3\alpha_j} - \zeta_{2\beta_j}^{-\alpha_j})}{\displaystyle\Pi^\dag_{j} ( {\bs{\alpha_k}})} (\ell\tau+1)^{-\frac12}\zeta_{24}^{-\ell}\mathcal E_1\left(\frac{\alpha_j}{\beta_j},\ell;\tau\right), 
\end{multline*}
where
\begin{equation}\label{def_mathcalE}
\mathcal E_1\left(\frac{\alpha}{\beta},\ell;\tau\right):=(\ell\tau+1)^{\frac12} \zeta_{24}^\ell q^{-\frac{1}{24}}e\left(\frac32 \frac{\alpha}{\beta} \right) - e\left(\frac{-S_\ell\tau}{24} \right)e\left(\frac32 \frac{\alpha}{\beta} \right).
\end{equation}
 
Thus in order to prove that $H_{n,S_\ell}(x)$ is analytic on $\mathbb{R} - \{\frac{-1}{\ell}\}$ it suffices to show that for each $1\leq j \leq n$,
\begin{align*}
G_{\alpha_j,\beta_j }(\tau) :=  F_{\alpha_j,\beta_j}(\tau) - \zeta_{24}^{-\ell}(\ell \tau +1)^{-\frac12}F_{\alpha_j,\beta_j}(S_\ell \tau) 
\end{align*}
is analytic on $\mathbb{R} - \{\frac{-1}{\ell}\}$. We establish this in Proposition \ref{prop_Habanalytic} below.

\begin{proposition} \label{prop_Habanalytic} Fix $1\leq j \leq n$ and set $(\alpha, \beta) := (\alpha_j, \beta_j)$.  With notation and hypotheses as above, we have that 
\begin{align*}  
G_{\alpha,\beta}(\tau) = \sqrt{3}\sum_{\pm}\mp e\left(\mp\frac16\right) \int_{\frac{1}{\ell}}^{i\infty}\frac{g_{\pm\frac13 + \frac12, \frac12-3\frac{\alpha}{\beta}}(3\rho)}{\sqrt{-i(\rho+\tau)}}d\rho,
\end{align*}
which is analytic on $\mathbb R - \left \{\frac{-1}{\ell}\right \}$.
\end{proposition}

\begin{proof}
Fix $1\leq j \leq n$ and set $(\alpha, \beta) := (\alpha_j, \beta_j)$. Define $m := \left[\frac{3\alpha}{\beta} \right] \in \mathbb{Z}$, $r\in (-\frac12, \frac12)$ so that $\frac{3\alpha}{\beta} =m + r$.  We note that $r\neq \pm \frac12$ since $\beta\neq 2$.  Using Proposition \ref{Rtransform} (1), we have that
\begin{equation}\label{eq:F_jtau}
F_{\alpha,\beta}(\tau) = q^{-\frac16} \sum_{\pm} \pm e \left(\frac{\mp r}{3} \right) e \left( \frac{\mp m}{3} \right) (-1)^{m} R\left( \pm \tau + r; 3\tau \right).
\end{equation}
Letting $\tau_\ell := -\frac{1}{\tau} - \ell$ we have $S_\ell \tau = \frac{-1}{\tau_\ell}$.  Using  Proposition \ref{Rtransform} (5) with $u=\frac{r}{3} \tau_\ell \mp \frac13$ and $\tau \mapsto \frac{\tau_\ell}{3}$ we see that
\begin{multline}\label{eq:h1}
R\left(r \mp \frac{1}{\tau_\ell} ; \frac{-3}{\tau_\ell} \right) = \\
\sqrt{ \frac{-i\tau_\ell}{3} } \cdot e\left(-\frac12\left(\frac{r\tau_\ell}{3} \mp \frac13\right)^2\left(\frac{3}{\tau_\ell}\right)\right)\left[h\left(\frac{r\tau_\ell}{3} \mp \frac13; \frac{\tau_\ell}{3} \right) - R\left(\frac{r\tau_\ell}{3} \mp \frac13; \frac{\tau_\ell}{3} \right) \right].
\end{multline}
Using Proposition \ref{Rtransform} parts (1) and (4) we see that $R\left(\frac{r\tau_\ell}{3} \mp \frac13; \frac{\tau_\ell}{3} \right) = \zeta_{24}^\ell R\left( \frac{-r}{3\tau} \mp \frac13 ; \frac{-1}{3\tau}  \right)$.  Then using Proposition \ref{Rtransform} (5)   with $u=\mp \tau - r$ and $\tau \mapsto 3\tau$ we obtain that 
\[ R\left(\frac{r\tau_\ell}{3} \mp \frac13; \frac{\tau_\ell}{3} \right) = \zeta_{24}^\ell \sqrt{-i(3\tau)} \cdot e\left(\frac{-\left(\mp\tau - r\right)^2}{6\tau} \right) \left[h\left( \mp\tau - r; 3\tau\right) - R\left( \mp\tau - r ; 3\tau \right) \right], \] 
which together with \eqref{eq:F_jtau} and \eqref{eq:h1}  gives
\begin{multline} \notag
F_{\alpha,\beta}(S_\ell \tau) = \\
e\left(\frac{1}{6\tau_\ell} \right) \sum_{\pm} \pm  e \left(\frac{\mp r}{3} \right) e \left( \frac{\mp m}{3} \right) (-1)^{m}
\sqrt{ \frac{-i\tau_\ell}{3} } \cdot e\left(-\frac12\left(\frac{r\tau_\ell}{3} \mp \frac13\right)^2\left(\frac{3}{\tau_\ell}\right)\right) \cdot \\
\left[h\left(\frac{r\tau_\ell}{3} \mp \frac13; \frac{\tau_\ell}{3} \right) - 
\zeta_{24}^\ell \sqrt{-i(3\tau)} \cdot e\left(\frac{-\left(\mp\tau - r\right)^2}{6\tau}\right) \left[h\left( \mp\tau - r; 3\tau\right) - R\left( \mp\tau - r ; 3\tau \right)\right] \right].
\end{multline}
By the definition of $r$ and $\ell$ we have that $\frac{r^2\ell}{6} \in \mathbb{Z}$.  Simplifying thus gives that
\begin{multline*}
F_{\alpha,\beta}(S_\ell \tau) = \sum_{\pm} \pm (-1)^{m} e\left( \frac{\mp m }{3}\right) e\left(\frac{r^2}{6\tau} \right) \sqrt{\frac{-i\tau_\ell}{3}} h\left( \frac{r\tau_\ell}{3} \mp \frac13 ; \frac{\tau_\ell}{3}\right) \\ 
 - \sum_{\pm}  \pm  (-1)^{m} e\left( \frac{\mp m}{3}\right) e\left(\frac{\mp r}{3} \right) q^{-\frac16} \zeta_{24}^\ell (\ell\tau + 1)^{\frac12} \cdot h\left(\mp \tau - r ; 3\tau \right) \\
 + \sum_{\pm}  \pm (-1)^{m} e\left( \frac{\mp m}{3}\right) e\left(\frac{\mp r}{3} \right) q^{-\frac16} \zeta_{24}^\ell (\ell\tau + 1)^{\frac12} \cdot R\left(\mp \tau - r ; 3\tau \right),
\end{multline*}
and so using Proposition \ref{Rtransform} (3) and the fact that $h(u;\tau)=h(-u;\tau)$ which comes directly from the definition of $h$   in \eqref{def_hmordell}, we see that
\begin{multline}\notag
G_{\alpha,\beta}(\tau) = q^{-\frac16} \sum_{\pm}\pm (-1)^{m} e\left( \frac{\mp m}{3}\right) e\left(\frac{\mp r}{3} \right)h\left(\pm \tau + r ; 3\tau \right) \\
 - \sum_{\pm}\pm (-1)^{m} e\left( \frac{\mp m}{3}\right) e\left(\frac{r^2}{6\tau}\right) \zeta_{24}^{-\ell}  \sqrt{\frac{i}{3\tau}} \cdot h\left( \frac{r\tau_\ell}{3} \mp \frac13 ; \frac{\tau_\ell}{3}\right).
\end{multline}
We now use Theorem \ref{thm_Zh2} to convert the $h$ functions into integrals.  Letting $a=\frac{\pm 1}{3}$, $b=-r$, and $\tau \mapsto 3\tau$ gives that 
\[ h\left(\pm \tau + r ; 3\tau \right) = -q^{\frac16} \zeta_6^{\mp 1} e\left(\frac{\pm r}{3} \right) \int_{0}^{i\infty} \frac{g_{\pm\frac13 + \frac12, \frac12 -r}(z) dz}{\sqrt{-i(z+3\tau)}}. \]
Letting $a=r$, $b=\frac{\pm 1}{3}$, and $\tau \mapsto \frac{\tau_\ell}{3}$ gives that 
\[h\left( \frac{r\tau_\ell}{3} \mp \frac13 ; \frac{\tau_\ell}{3}\right) = -e\left(\frac{-r^2}{6\tau} \right) e\left(\frac{\mp r}{3} \right) e\left(\frac{-r}{2} \right) \int_{0}^{i\infty} \frac{g_{r + \frac12, \pm\frac13 + \frac12}(z) dz}{\sqrt{-i\left(z+\frac{\tau_\ell}{3} \right)}}.\] 
Thus 
\begin{multline}\notag
G_{\alpha,\beta}(\tau) = -\sum_{\pm}  \pm  \zeta_6^{\mp1}(-1)^{m} e\left( \frac{\mp m}{3}\right) \int_{0}^{i\infty} \frac{g_{\pm\frac13 + \frac12, \frac12 -r}(z) dz}{\sqrt{-i(z+3\tau)}} \\
+\sum_{\pm}  \pm \zeta_{24}^{-\ell}(-1)^{m} e\left( \frac{\mp m}{3}\right)e\left(\frac{\mp r}{3} \right) e\left(\frac{-r}{2} \right) \sqrt{\frac{i}{3\tau}} \int_{0}^{i\infty} \frac{g_{r+ \frac12, \pm\frac13 + \frac12}(z) dz}{\sqrt{-i\left(z+\frac{\tau_\ell}{3} \right)}}.
\end{multline}

By a simple change of variables (let $z=\frac{\ell}{3} - \frac{1}{z}$) we can write 
\begin{equation} \label{eq:int_convert}
\int_{0}^{i\infty} \frac{g_{r + \frac12, \pm\frac13 + \frac12}(z) dz}{\sqrt{-i\left(z+\frac{\tau_\ell}{3} \right)}} = -\sqrt{-3\tau} \int_{\frac{3}{\ell}}^{0} \frac{g_{r + \frac12, \pm\frac13 + \frac12}\left(\frac{\ell}{3} - \frac{1}{z} \right) dz}{z^{\frac32}\sqrt{-i(z+3\tau)}}.
\end{equation}
Moreover, using  Proposition \ref{prop_Zg} we can convert
\begin{multline} \label{eq:g_convert}
g_{r + \frac12, \pm\frac13 + \frac12}\left(\frac{\ell}{3} - \frac{1}{z} \right) = \zeta_{24} ^\ell \cdot  g_{r-\frac12, \pm\frac13 + \frac12}\left( \frac{-1}{z}\right)  \\ 
= -\zeta_{24} ^\ell e\left(\frac18 \right) e\left(\frac{\mp1}{6} \right) e\left(\frac{\pm r}{3} \right) e\left(\frac{r}{2} \right)  z^{\frac32} \cdot  g_{\pm\frac13 + \frac12, \frac12 - r}(z).
\end{multline}
Thus by \eqref{eq:int_convert} and \eqref{eq:g_convert} we have that
\begin{align}
G_{\alpha,\beta}(\tau) &=- \sum_{\pm} \pm \zeta_6^{\mp1}(-1)^{m} e\left( \frac{\mp m}{3}\right) \int_{0}^{i\infty} \frac{g_{\pm\frac13 + \frac12, \frac12 -r}(z) dz}{\sqrt{-i(z+3\tau)}} \notag \\  
 &\hspace{.5in}-   \sum_{\pm} \pm  \zeta_6^{\mp1}(-1)^{m } e\left( \frac{\mp m}{3}\right)  \int_{\frac{3}{\ell}}^{0}\frac{g_{\pm\frac13 + \frac12, \frac12 -r}(z) dz}{\sqrt{-i(z+3\tau)}} \notag \\ 
&= -\sum_{\pm}  \pm \zeta_6^{\mp1}(-1)^{m} e\left( \frac{\mp m}{3}\right)  \int_{\frac{3}{\ell}}^{i\infty}\frac{g_{\pm\frac13 + \frac12, \frac12 -r}(z) dz}{\sqrt{-i(z+3\tau)}}. \label{Hab_integral}
\end{align}
 
To complete the proof, one can deduce from Proposition \ref{prop_Zg} (2) that for $m\in\Z$,
\begin{align*}
g_{a,b}(\tau)=e(m a)g_{a,b-m}(\tau).
\end{align*}
Applying this to \eqref{Hab_integral} with a direct calculation gives us 
\[
G_{\alpha,\beta}(\tau) = \sqrt{3}\sum_{\pm}\mp e\left(\mp\frac16\right) \int_{\frac{1}{\ell}}^{i\infty}\frac{g_{\pm\frac13 + \frac12, \frac12-3\frac{\alpha}{\beta}}(3z)}{\sqrt{-i(z+\tau)}}dz,
\]
 which is analytic on $\mathbb{R} - \{\frac{-1}{\ell}\}$ as desired. 
\end{proof}
 
\section{Conclusion}

We have proven that when we restrict to vectors $\boldsymbol{\zeta_n}$ which contain distinct roots of unity, the mock modular form $q^{-\frac{1}{24}} R_n(\boldsymbol{\zeta_n};q)$ is also a quantum modular form.  To consider the more general case where we allow roots of unity in $\boldsymbol{\zeta_n}$ to repeat, the situation is significantly more complicated.  In this setting, as shown in \cite{F-K}, the nonholomorphic completion of $q^{-\frac{1}{24}} R_n(\boldsymbol{\zeta_n};q)$ is not modular, but is instead a sum of two (nonholomorphic)     modular forms of different weights.  We will address this more general case in a forthcoming paper \cite{FJKS}.

\end{document}